\documentclass[reqno,3p,12pt]{elsarticle}
\usepackage[latin1]{inputenc}
\usepackage{amsmath}
\usepackage{amsthm}
\usepackage{amsfonts}
\usepackage{amssymb}
\usepackage{enumerate}

\newtheorem{thm}{Theorem}
\newtheorem{prop}[thm]{Proposition}
\newtheorem{cor}[thm]{Corollary}

\DeclareMathOperator{\charak}{char}
\DeclareMathOperator{\lcm}{lcm}
\DeclareMathOperator{\im}{im}
\DeclareMathOperator{\res}{res}
\DeclareMathOperator{\Hom}{Hom}
\DeclareMathOperator{\Gal}{Gal}
\DeclareMathOperator{\Br}{Br}
\DeclareMathOperator{\ind}{ind}

\newcommand{\ndiv}{\nmid}
\newcommand{\set}[1]{\{#1\}}
\newcommand{\resm}[2]{\res_{#1\to #2}}
\newcommand{\wt}{\widetilde}
\newcommand{\auss}[1]{{`#1'}}
\newcommand{\ovl}[1]{\overline{#1}}
\newcommand{\field}[1]{\mathbb{#1}}
\newcommand{\Q}{\field{Q}}
\newcommand{\N}{\field{N}}
\newcommand{\Z}{\field{Z}}
\renewcommand{\P}{\field{P}}

\begin{document}
\begin{frontmatter}
\title{Galois subfields of inertially split division algebras}
\author{Timo Hanke}
\address{
Lehrstuhl D für Mathematik\\
RWTH Aachen\\
Templergraben 64\\
D-52062 Aachen\\
Germany
}
\ead{hanke@math.rwth-aachen.de}
\date{\today}
\begin{abstract}
Let $D$ be a valued division algebra, finite-dimensional over its center $F$.
Assume $D$ has an unramified splitting field.
The paper shows that if $D$ contains a maximal subfield which is Galois over $F$
(i.e.\ $D$ is a crossed product) then the residue division algebra $\ovl D$
contains a maximal subfield which is Galois over the residue field $\ovl F$.
This theorem captures an essential argument of previously known noncrossed product proofs
in the more general language of noncommutative valuations.
The result is particularly useful in connection with explicit constructions.
\end{abstract}

\end{frontmatter}

\begin{quotation}
  \small NOTICE: this is the author's version of a work that was accepted for
publication in Journal of Algebra. Changes resulting from the publishing
process, such as peer review, editing, corrections, structural formatting, and
other quality control mechanisms may not be reflected in this document. Changes
may have been made to this work since it was submitted for publication. 
A definitive version will be subsequently published in Journal of Algebra (2011), doi:10.1016/j.jalgebra.2011.08.019.
\end{quotation}
\section{Introduction}

Let $(F,v)$ be an arbitrary valued field.
By a valuation $v$ on $F$ we mean a Krull valuation,
i.e.\ there is no restriction on the rank or divisibility of the value group.
Let $D$ be a finite-dimensional central $F$-division algebra.
Assume $v$ extends to a valuation on $D$
and denote this extension also by~$v$.
By a valuation on $D$ we mean a valuation in the sense of Schilling \cite{schilling:th-of-val}, i.e.\ $v$ corresponds to an invariant valuation ring of $D$. 
We call $(D,v)$ {\em inertially split} if there exists an unramified field extension $L/F$ such that $L$ splits $D$.\footnote{
In other words, $D$ is inertially split if and only if $D$ is split by the maximal unramified extension $F^h_{nr}$ of the Henselization $F^h$ of $(F,v)$.}
Let $\ovl F, \ovl D$ denote the residue field of $(F,v)$ and the residue division algebra of $(D,v)$, respectively.
A {\em maximal subfield} of $D$ means a commutative subfield of $D$ 
which is maximal with respect to inclusion.

\begin{thm}\label{thm:main}
Let $(D,v)$ be inertially split, $F=Z(D)$.
If $D$ contains a maximal subfield Galois over $F$
then 
$\ovl D$ contains a maximal subfield Galois over $\ovl F$.
The converse holds if $v$ is Henselian.
\end{thm}

A division algebra $D$ is called a {\em crossed product} if it contains a maximal subfield Galois over its center, otherwise a {\em noncrossed product}.
The existence as well as the construction of noncrossed products has spawned research interest from around 1930, when the question of existence first arose, to the present day.
Since Amitsur's existence proof in 1972 
more examples were found or constructed by various authors and by various methods.
In this context, 
Theorem \ref{thm:main} is regarded as a noncrossed product criterion for valued division algebras
which is formulated purely ``on the residue level''.
Note that the condition \auss{$\ovl D$ does not contain a maximal subfield Galois over $\ovl F$}
is weaker than \auss{$\ovl D$ is a noncrossed product}
because $\ovl F$ is in general a strict subfield of the center of $\ovl D$.

Valuations on division rings have played a crucial role in works on noncrossed products 
(explicitly or implicitly, as \cite[\S5]{wadsworth:survey} points out).
So it comes as no surprise that special cases of Theorem~\ref{thm:main} can be identified as a key ingredient in previously known noncrossed product proofs 
(even where valuations on division rings are not explicitly mentioned),
starting with Brussel \cite{brussel:noncr-prod}.
In fact, it was the motivation of the present paper to provide an abstraction of this useful argument 
in the language of noncommutative valuations, which seems to capture its essence best.
The precise relation of Theorem \ref{thm:main} to the relevant literature is discussed in several remarks.
We also point out its usefulness in connection with explicit noncrossed product constructions.

For the most part the paper contains condensed material from the authors thesis \cite{hanke:thesis},
described in its historical context up to recent work in \cite{hanke-sonn:location}.

\section{The main theorem}

Let $D$ be a division algebra, finite-dimensional over its center $Z(D)$.
Jacob and Wadsworth show in \cite[Thm.~5.15(b)]{jacob-wadsworth:div-alg-hensel-field} that 
if $D$ is inertially split\footnote{We may assume the valuation to be Henselian by passing to the Henselization of $D$ in the sense of \cite{morandi:henselization-div-alg}.} and
contains a maximal subfield $L$ Galois over $Z(D)$
then $\ovl D$ contains the maximal subfield $\ovl LZ(\ovl D)$,
which is normal over $\ovl{Z(D)}$.
They conclude further that $\ovl D$ is a crossed product using a 
result of Saltman, \cite[Lem.~3]{saltman:noncr-prod-small-exp}. 
This section points out that replacing \cite[Lem.~3]{saltman:noncr-prod-small-exp} by a slightly stronger statement (see Proposition \ref{prop:galsf} below) leads to Theorem \ref{thm:main}.
In the preface of \cite{hanke:thesis} Theorem \ref{thm:main} appears as ``Noncrossed Product Criterion'', and in the text as Theorem 5.20.

\begin{prop}\label{prop:galsf}
Let $F\subseteq Z(D)$ be a subfield such that $Z(D)/F$ is finite and separable.
If $D$ contains a maximal subfield that is normal over $F$
then $D$ also contains a maximal subfield that is Galois over $F$.
\end{prop}
The case $F=Z(D)$ is \cite[Lem.~3]{saltman:noncr-prod-small-exp}.
An inspection of the proof reveals that it also handles the case when $Z(D)/F$ is finite and separable.
Further details can be found in the appendix of \cite{hanke:thesis}.
Note that this is essentially a statement about $p$-algebras.

\begin{thm}\label{thm:main2}
Let $(D,v)$ be inertially split, $F\subseteq Z(D)$ a finite degree subfield.
Suppose $Z(D)/F$ and $Z(\ovl D)/\ovl F$ are Galois.
If $D$ contains a maximal subfield which is Galois over $F$
then $\ovl D$ contains a maximal subfield which is Galois over $\ovl F$.
The converse holds if $v$ is Henselian.
\end{thm}
If $F=Z(D)$ then the hypothesis \auss{$Z(\ovl D)/\ovl F$ is Galois} is automatically satisfied
(\cite[Lem.~5.1]{jacob-wadsworth:div-alg-hensel-field}).
Thus, Theorem \ref{thm:main} is a special case of Theorem~\ref{thm:main2}.
The generality of Theorem~\ref{thm:main2} lends itself to an iterated application for composite valuations, 
which is possible but not pursued further in this paper.

\begin{proof}[Proof of Theorem \ref{thm:main2}]
Let $L$ be a maximal subfield of $D$ which is Galois over $F$.
Following the lines of proof of \cite[Thm.~5.15]{jacob-wadsworth:div-alg-hensel-field}, $\ovl L Z(\ovl D)$ is a maximal subfield of $\ovl D$.
Since $L/F$ is Galois, $\ovl L/\ovl F$ is normal (\cite[p.\ 107]{endler:val-th}).
Since $Z(\ovl D)/\ovl F$ is Galois by hypothesis, $\ovl L Z(\ovl D)/\ovl F$ is normal.
Proposition \ref{prop:galsf} shows that $\ovl D$ contains a maximal subfield Galois over $\ovl F$.

Conversely, let $L$ be a maximal subfield of $\ovl D$ which is Galois over $\ovl F$.
Assume $v$ is Henselian.
By \cite[Thm.~2.9]{jacob-wadsworth:div-alg-hensel-field},
$D$ contains an inertial lift $\wt L$ of $L$ over $F$
and $\wt L/F$ is Galois.
Hence $M:=Z(D)\wt L$ is Galois over $F$ and inertial over $Z(D)$.
Since $\ovl M=L$ we have $[M:Z(D)]=[L:\ovl{Z(D)}]=\ind D$ by the index formula in 
\cite[Thm.~5.15]{jacob-wadsworth:div-alg-hensel-field}.
Thus, $M$ is maximal in $D$. 
\end{proof}

\section{Existence of inertially split division algebras}

Assume $(F,v)$ is Henselian.
The inertially split division algebras over $F$ then form a subgroup of the Brauer group;
this subgroup is the relative Brauer group $\Br(F_{nr}/F)$,
where $F_{nr}$ is the maximal unramified extension of $F$.
Let $G=\Gal(\ovl F_{sep}/\ovl F)$, $\Gamma=v(F^*)$, $\Delta$ the divisible hull of $\Gamma$.
Any continuous homomorphism $\chi\in\Hom_c(G,\Delta/\Gamma)$ has a finite image when $\Delta/\Gamma$ is equipped with the discrete topology because $G$ is profinite.
Thus, denoting the fixed field of $\ker(\chi)$ by $\ovl F(\chi)$,
the extension $\ovl F(\chi)/\ovl F$ is finite abelian.
Moreover, $\Gal(\ovl F(\chi)/\ovl F)$ has rank
at most the rank of $v$.
Let $\res_{\ovl F\to\ovl F(\chi)}:\Br(\ovl F)\to\Br(\ovl F(\chi)),\alpha\mapsto\alpha^\chi$ denote the restriction map.
There is a (noncanonical) exact sequence, due to Witt~\cite{witt:schiefkoerper} and Scharlau~\cite{scharlau:br-henselkoerper}:
\begin{equation}\label{eq:seq}
 0\to\Br(\ovl F)\to\Br(F_{nr}/F)\stackrel{\gamma}{\to}\Hom_c(G,\Delta/\Gamma)\to 0.
\end{equation}
The interpretation of the cohomological data in this sequence has been extensively studied 
by Jacob and Wadsworth in \cite[p.\ 154ff]{jacob-wadsworth:div-alg-hensel-field}:
Let $(D,v)$ be inertially split with center $F$.
Then $Z(\ovl D)=\ovl F(\gamma[D])$ (\cite[Thm.~5.6(b)]{jacob-wadsworth:div-alg-hensel-field})
and $\ind D=\ind\ovl D\cdot[\ovl F(\gamma[D]):\ovl F]$
(\cite[Thm.~5.15(a)]{jacob-wadsworth:div-alg-hensel-field}).
There exists a splitting map $\delta$ for $\gamma$ such that
\begin{equation}\label{eq:delta}
\textrm{$\ovl D$ is a field for each $[D]\in\im(\delta)$}
\end{equation}
(\cite[Rem.~5.9(ii)]{jacob-wadsworth:div-alg-hensel-field}).
For any such $\delta$,
if $[D]=\alpha+\delta(\chi)$ with $\alpha\in\Br(\ovl F)$ and $\chi\in\Hom_c(G,\Delta/\Gamma)$ then
$[\ovl D]=\alpha^\chi$
(following the proof of \cite[Thm.~5.15(a)]{jacob-wadsworth:div-alg-hensel-field}). 
Proposition \ref{prop:ex} below is an immediate consequence of these facts.
Suppose $E$ is a finite-dimensional division algebra over $\ovl F$ (not necessarily central).

\begin{prop}\label{prop:ex}
$E$ is the residue of an $F$-central inertially split $(D,v)$
if and only if 
$Z(E)=\ovl F(\chi)$ for some $\chi\in\Hom_c(G,\Delta/\Gamma)$ 
and $[E]\in\im(\resm{\ovl F}{Z(E)})$.
\end{prop}

An $F$-central inertially split $(D,v)$ with residue division algebra $E$ will be called an {\em $F$-central lift of $E$}
($v$ is not required to be Henselian for this definition). 
Proposition \ref{prop:ex} appears in the preface of \cite{hanke:thesis} as ``Lift Theorem'',
and in the text as Theorem~5.25.
All $F$-central lifts of $E$ are obtained by taking, 
for a fixed $\delta$ satisfying \eqref{eq:delta},
the underlying division algebras of $\alpha+\delta(\chi)$ 
where $\chi$ runs over all $\chi\in\Hom_c(G,\Delta/\Gamma)$ with $\ovl F(\chi)=Z(E)$
and $\alpha$ runs over all $\alpha\in\Br(\ovl F)$
with $\alpha^\chi=[E]$.
There is, however, no canonical $F$-central lift because there is no canonical choice for $\alpha$.\footnote{
As a consequence, for instance, the exponent of an $F$-central lift of $E$ is in general not determined by $E$.
Note that $\delta$ can be chosen order preserving by \cite[Ex.~4.3]{jacob-wadsworth:div-alg-hensel-field}. 
In that case,
following the proof of \cite[Thm.~5.15(a)]{jacob-wadsworth:div-alg-hensel-field},
one derives $\exp D=\lcm(\exp\alpha,\exp\chi)$.
This expression is in general not determined by $E$;
see \cite[Ex.\ 5.13]{hanke:thesis} for an example.
} 
Translating Theorem \ref{thm:main} in terms of $\alpha$ and $\chi$ yields

\begin{cor}\label{cor:brussel}
Suppose $\delta$ satisfies \eqref{eq:delta}. 
Given a pair $(\alpha,\chi)\in\Br(\ovl F)\times\Hom_c(G,\Delta/\Gamma)$,
the division algebra $D$ underlying $\alpha+\delta(\chi)$ is a crossed product if and only if 
there is a Galois extension $M/\ovl F$ of
degree $\ind D$ that contains $\ovl F(\chi)$ and splits~$\alpha$.
\end{cor}
\begin{proof}
Let $M\supseteq \ovl F(\chi)$.
Since $\alpha^\chi=[\ovl D]$, $M$ splits $\alpha$ if and only if $M$ splits $\ovl D$.
Since $\ind D=\ind\ovl D\cdot[\ovl F(\chi):\ovl F]$,
$M$ is a maximal subfield of $\ovl D$ if and only if $[M:\ovl F]=\ind D$.
The assertion is thus equivalent to Theorem \ref{thm:main}.
\end{proof}

The existence of noncrossed products over $F$ can thus be shown by exhibiting pairs $(\alpha,\chi)$ for which no $M$ as in Corollary \ref{cor:brussel} exists.
This approach is original to \cite{brussel:noncr-prod} and has been followed in \cite{hanke-sonn:location}
(see remark b) in \S\ref{sec:rank1} below).

\section{Construction of inertially split division algebras}
Let $(F,v)$ be an arbitrary valued field.
Let $E$ be a finite-dimensional division algebra over $\ovl F$.
Assume $Z(E)=\ovl F(\chi)$ for some $\chi\in\Hom_c(G,\Delta/\Gamma)$ 
and $[E]\in\im(\resm{\ovl F}{Z(E)})$.
Proposition~\ref{prop:ex} states the existence of an $F$-central lift of $E$ (over any Henselian valued field $(F,v)$ with residue field $\ovl F$)
but does not give a direct construction of the lift.
Here, a construction is {\em direct} if we are not required to pass from some representative of the class $\alpha+\delta(\chi)$ to its underlying division algebra.
This section points out that if one is allowed to choose $F$
then a lift can be obtained by Tignol's construction  (\cite{tignol:gen-cr-prod}) of a {\em generic abelian extension} of $E$.
The construction is as follows,
similar to the generic abelian crossed products of  
Amitsur and Saltman in \cite{amitsur-saltman:gen-abel-cr-prod}.

Since $[E]$ lies in the image of $\resm{\ovl F}{Z(E)}$,
there is a central-simple $\ovl F$-algebra $A$ and an $\ovl F$-algebra embedding of $E$ into $A$
such that $E$ is the centralizer of $Z(E)$ in $A$.
Suppose $A$ as well as the embedding of $E$ are explicitly given,
and identify $E$ with its image in $A$.
Choose a minimal set of generators $(\sigma_1,\ldots,\sigma_r)$ 
of the abelian group $\Gal(Z(E)/\ovl F)$.
By the Skolem-Noether theorem, choose for each $1\leq i\leq r$ 
an element $z_i\in A^*$ 
so that $z_iaz_i^{-1}=\sigma_i(a)$ for all $a\in Z(E)$.
Let $E[x_1,\ldots,x_r;z]$ denote the twisted polynomial ring defined by the relations
\[  x_ia=(z_iaz_i^{-1})x_i,\quad x_ix_j=(z_iz_jz_i^{-1}z_j^{-1})x_jx_i, \] 
for all $a\in E$, $1\leq i,j\leq r$.
Note that $z_iaz_i^{-1}$ and $z_iz_jz_i^{-1}z_j^{-1}$ lie in $E$.
Let $E(x_1,\ldots,x_r;z)$ denote the ring of central quotients of $E[x_1,\ldots,x_r;z]$.
As in the proof of \cite[Thm.~2.3]{tignol:gen-cr-prod} one may verify 

\begin{thm}\label{thm:tignol}
$E(x_1,\ldots,x_r;z)$ is an $F$-central lift of $E$
where $F$ is isomorphic to the rational function field $\ovl F(t_1,\ldots,t_r)$.
\end{thm}

Theorem \ref{thm:tignol} appears in \cite{hanke:thesis} as Theorem 11.14.
Theorems \ref{thm:main} and \ref{thm:tignol} allow to construct explicit examples of noncrossed products by exhibiting division algebras $E$ as above that do not contain a maximal subfield Galois over $\ovl F$.
Here, an example is {\em explicit} if the structure constants of the noncrossed product are known.
This approach is original to \cite{hanke:thesis} and has been followed in \cite{hanke:expl-ex},\cite{hanke:laurent-noncr}.

\section{Rank $1$ valuations}\label{sec:rank1}

Assume $v$ has rank $1$.
We conclude with remarks and references about this important case.

\begin{enumerate}[a)]
\item
Every tame division algebra (in the sense of Jacob-Wadsworth \cite[\S6]{jacob-wadsworth:div-alg-hensel-field}) is inertially split.
This can be seen as follows: By \cite[Lem.~6.2]{jacob-wadsworth:div-alg-hensel-field}, any tame $D$ has a decomposition $D\sim S\otimes T$ where $S$ is inertially split and $T$ is tame and totally ramified.
In the rank $1$ case, $T$ must be trivial because there are no non-trivial tame and totally ramified division algebras.

\item
If $v$ is discrete (i.e.\ $\Gamma=\Z$) then \eqref{eq:seq} is split exact and the splitting homomorphism $\delta$ satisfies \eqref{eq:delta}.
Thus, one recovers Witt's theorem
\begin{equation}\label{eq:witt}
\Br(F_{nr}/F)\cong\Br(\ovl F)\oplus\Hom_c(G,\Q/\Z).
\end{equation}
Moreover, in Corollary \ref{cor:brussel},
$(\alpha,\chi)$ is the decomposition of $D$ according to \eqref{eq:witt}.
In this form, Corollary \ref{cor:brussel} is not new; it is original to Brussel \cite{brussel:noncr-prod} 
where it appears as ``Corollary'' on p.~381 (for $\charak \ovl F=0$).
It also appears in \cite[\S6]{hanke-sonn:location} as ``Brussel's Lemma''.
Note that $\Br(F_{nr}/F)=\Br(F)$ if $\ovl F$ is perfect.

\item
Replacing in \cite{hanke-sonn:location} the use of Corollary \ref{cor:brussel} by Theorem \ref{thm:main} leads to the following reformulation of the main result of \cite{hanke-sonn:location}:
Suppose $\ovl F$ is a global field and $v$ is discrete.
For any finite cyclic extension $K/\ovl F$ and any $m\in\N$
let $B_m(K)$ denote the set of all inertially split $F$-central division algebras with $Z(\ovl D)=K$ and $\ind\ovl D=m$.
The sets $B_m(K)$ form a partition of $\Br(F_{nr}/F)$. 
There is a formal product\footnote{called a {\em supernatural number} or {\em Steinitz number}} $b(K)=\prod_{p\in\P}p^{b_p}$ with $b_p\in\N\cup\set{\infty}$ such that
\[\textrm{
$B_m(K)$ consists entirely of crossed products if and only if $m|b(K)$.}\]
If $m\ndiv b(K)$ then $B_m(K)$ contains infinitely many noncrossed products.
The numbers $b(K)$ are computed in \cite{hanke-sonn:location} by a precise formula in terms of the number of roots of unity in $K$
and a measure for the embeddability of the cyclic extension $K/\ovl F$ into larger cyclic extensions.
For example, if $\ovl F=\Q$ then $b_2(\Q(\sqrt{-1}))=2$,
hence $B_2(\Q(\sqrt{-1}))$ and $B_4(\Q(\sqrt{-1}))$ consist entirely of crossed products while $B_8(\Q(\sqrt{-1}))$ contains infinitely many noncrossed products. 
All elements of $B_8(\Q(\sqrt{-1}))$ have index $16$.

\item
The result formulated in c) holds not only for discrete $v$.
In the non-discrete case, however, certain sets $B_m(K)$ can be empty.
In order to get a partition of $\Br(F_{nr}/F)$ we have to restrict ourselves to the cyclic extensions $K/\ovl F$ of the form $K=\ovl F(\chi)$ for some $\chi\in\Hom_c(G,\Delta/\Gamma)$.
\end{enumerate}

\def\cprime{$'$}

\end{document}